\documentclass[reqno]{amsart}
\usepackage{amssymb,setspace}
\usepackage{ifpdf}
\ifpdf
 \usepackage[hyperindex]{hyperref}%,pagebackref
\else
 \expandafter\ifx\csname dvipdfm\endcsname\relax
 \usepackage[hypertex,hyperindex]{hyperref}
 \else
 \usepackage[dvipdfm,hyperindex]{hyperref}
 \fi
\fi
\theoremstyle{theorem}
\newtheorem{thm}{Theorem}[section]
\newtheorem{lem}{Lemma}[section]
\newtheorem{conj}{Conjecture}[section]
\newtheorem{cor}{Corollary}[section]
\theoremstyle{definition}
\newtheorem{dfn}{Definition}[section]
\theoremstyle{remark}

\DeclareMathOperator{\td}{d\mspace{-2mu}}
\DeclareMathOperator{\arcsinh}{arcsinh}
\allowdisplaybreaks[4]
\numberwithin{equation}{section}

\begin{document}

\title[Convexity of generalized sine and generalized hyperbolic sine]
{Geometric convexity of the generalized sine and the generalized hyperbolic sine}

\author[W.-D. Jiang]{Wei-Dong Jiang}
\address[Jiang]{Department of Information Engineering, Weihai Vocational University\\ Weihai City, Shandong Province, 264210, China}
\email{\href{mailto: Wei-Dong Jiang <jackjwd@163.com>}{jackjwd@163.com}}

\author[F. Qi]{Feng Qi}
\address[Qi]{Department of Mathematics, School of Science, Tianjin Polytechnic University\\ Tianjin City, 300387, China}
\email{\href{mailto: F. Qi <qifeng618@gmail.com>}{qifeng618@gmail.com}, \href{mailto: F. Qi <qifeng618@hotmail.com>}{qifeng618@hotmail.com}, \href{mailto: F. Qi <qifeng618@qq.com>}{qifeng618@qq.com}}
\urladdr{\url{http://qifeng618.wordpress.com}}



\begin{abstract}
In the paper, the authors prove that the generalized sine function $\sin_{p,q}(x)$ and the generalized hyperbolic sine function $\sinh_{p,q}(x)$ are geometrically concave and geometrically convex, respectively. Consequently, the authors verify a conjecture posed in the paper ``B. A. Bhayo and M. Vuorinen, \emph{On generalized trigonometric functions with two parameters}, J. Approx. Theory \textbf{164} (2012), no.~10, 1415\nobreakdash--1426; Available online at \url{http://dx.doi.org/10.1016/j.jat.2012.06.003}''.
\end{abstract}

\thanks{This work was supported in part by the Project of Shandong Province Higher Educational Science and Technology Program under grant No.~J11LA57}

\maketitle

\section{Introduction}
It is well known from calculus that
$$
\arcsin x=\int_0^x \frac{1}{(1-t^2)^{1/2}}\td t
$$
for $0\le x\le1$ and
$$
\frac{\pi}{2}=\arcsin1=\int_0^1 \frac{1}{(1-t^2)^{1/2}}\td t.
$$
We can define the sine function on $\bigl[0,\frac\pi2\bigr]$ as the inverse of the arcsine function and extend it to $(-\infty,\infty)$.
\par
Let $1<p<\infty$. the arcsine can be generalized as
$$
\arcsin_px=\int_0^x \frac{1}{(1-t^p)^{1/p}}\td t, \quad 0\le x\le1
$$
and
$$
\frac{\pi_p}{2}=\arcsin_p1=\int_0^1 \frac{1}{(1-t^p)^{1/p}}\td t.
$$
The inverse of the function $\arcsin_p$ on $\bigl[0,\frac{\pi_p}2\bigr]$ is called the generalized sine function and denoted by $\sin_p$.
By standard extending procedures as the sine function done, we may obtain a differentiable function on the whole real line $(-\infty,\infty)$, which coincides with the sine when $p=2$.
It is easy to see that the function $\sin_p$ is strictly increasing and concave on $\bigl[0,\frac{\pi_p}2\bigr]$. Similarly, we can define the generalized cosine, tangent, hyperbolic functions and their inverses.
\par
For $p,q>1$, let
\begin{equation*}
F_{p,q}(x)=\int_{0}^{x}(1-t^q)^{-1/p}\td t,\quad x\in[0,1]
\end{equation*}
and
\begin{equation*}
\frac{\pi_{p,q}}{2}=\int_{0}^{1}(1-t^q)^{-1/p}\td t.
\end{equation*}
Then $F_{p,q}:[0,1]\to \bigl[0,\frac{\pi_{p,q}}2\bigr]$ is an increasing homeomorphism. We denote this function $F_{p,q}$ by $\arcsin_{p,q}$. Thus, its inverse
\begin{equation*}
\sin_{p,q}=F_{p,q}^{-1}
\end{equation*}
is defined on the interval $\bigl[0,\frac{\pi_{p,q}}2\bigr]$. By a similar extension to the sine function, we can find a differentiable function $\sin_{p,q}$ defined on $\mathbb{R}$. We call $\sin_{p,q}$ the generalized $(p,q)$-sine function.

We can also define
$$
\arccos_{p,q}x=\arcsin_{p,q}\bigl[(1-x^p)^{1/q}\bigr],
$$
see~\cite{Jiang-Qi-conj2, Jiang-Qi-conj4}, and the inverse of the generalized $(p,q)$-hyperbolic sine
function
\begin{equation*}
\arcsinh_{p,q}(x)=\int_{0}^{x}(1+t^q)^{-1/p}\td t,\quad x\in(0,\infty).
\end{equation*}
Their inverse functions are
\begin{equation*}
\sin_{p,q}:\biggl(0,\frac{\pi_{p,q}}2\biggr)\to (0,1),\quad \cos_{p,q}:\biggl(0,\frac{\pi_{p,q}}2\biggr)\to (0,1),
\end{equation*}
and
$$
\sinh_{p,q}:\bigl(0,m^*_{p,q}\bigr)\to (0,\infty),
$$
where
$$
m^*_{p,q}=\int_{0}^{\infty}(1+t^q)^{-1/p}\td t.
$$
\par
When $p=q$, the $(p,q)$-functions $\sin_{p,q}$, $\cos_{p,q}$,
$\sinh_{p,q}$, $\arcsin_{p,q}$, $\arccos_{p,q}$, and $\arcsinh_{p,q}$
reduce to $p$-functions $\sin_{p}$, $\cos_{p}$, $\sinh_{p}$,
$\arcsin_{p}$, $\arccos_{p}$, and $\arcsinh_{p}$ respectively. See~\cite{Jiang-Qi-conj3, Jiang-Qi-conj5, Jiang-Qi-conj7}. In particular, when $p=q=2$, the $(p,q)$-functions
become our familiar trigonometric and hyperbolic functions.
\par
Recently, the generalized trigonometric and hyperbolic functions, including $(p,q)$-functions and $p$-functions, have been studied by many mathematicians from
different points of view. See~\cite{Jiang-Qi-conj1, Jiang-Qi-conj6, Jiang-Qi-conj8, Jiang-Qi-conj9}. In~\cite{Jiang-Qi-conj4}, the authors gave basic properties of the generalized $(p,q)$-trigonometric
functions. In~\cite{Jiang-Qi-conj5}, the authors generalized some classical
inequalities for trigonometric and hyperbolic functions, such as
Mitrinovi\'c-Adamovi\'c inequality, Lazarevi\'c's
inequality, Huygens-type inequalities, and Wilker-type inequalities.
\par
In~\cite{Jiang-Qi-conj2}, the authors found that the functions $\arcsin_{p,q}$ and
$\arcsinh_{p,q}$ can be expressed in terms of Gaussian
hypergeometric functions. Applying the vast available information
about hypergeometric functions, some remarkable properties and
inequalities for generalized trigonometric and hyperbolic functions
are obtained. Moreover, they raised the following conjecture.

\begin{conj}[{\cite[p.~1421, Conjecture~2.11]{Jiang-Qi-conj2}}]
If $p,q\in(1,\infty)$ and $r,s\in(0,1)$, then
\begin{equation}\label{conj-ineq-1}
\sin_{p,q}\sqrt{rs}\,\ge \sqrt{\sin_{p,q}r\sin_{p,q}s}\,
\end{equation}
and
\begin{equation}\label{conj-ineq-2}
\sinh_{p,q}\sqrt{rs}\,\le \sqrt{\sinh_{p,q}r\sinh_{p,q}s}.
\end{equation}
\end{conj}

The main purpose of this paper is to discover the geometric convexity of $\sin_{p,q}(x)$ and $\sinh_{p,q}(x)$ and to give an affirmative answer to the above stated~\cite[p.~1421, Conjecture~2.11]{Jiang-Qi-conj2}.
\par
Our main results may be formulated in the following theorem.

\begin{thm}\label{th1.1}
Let $p,q\in(1,\infty)$ and $r,s,x\in(0,1)$. Then
\begin{enumerate}
\item
the function $\sin_{p,q}(x)$ is geometrically concave function;
\item
the function $\sinh_{p,q}(x)$ is geometrically convex function.
\end{enumerate}
\end{thm}

\begin{cor}\label{cor1.1}
Let $p,q\in(1,\infty)$ and $r,s,x\in(0,1)$. Then the inequalities~\eqref{conj-ineq-1} and~\eqref{conj-ineq-2} and the inequalities
\begin{equation}
\sin_{p}\sqrt{rs}\,\ge \sqrt{\sin_{p}r\sin_{p}s}
\end{equation}
and
\begin{equation}
\sinh_{p}\sqrt{rs}\,\le \sqrt{\sinh_{p}r\sinh_{p}s}
\end{equation}
are all valid.
\end{cor}

\section{A definition and lemmas}

For proving our main results, we need the following definition and lemmas.

\begin{dfn}\label{dfn1.1}
Let $I\subseteq (0,+\infty)$ be an interval and $f:I\to (0,\infty)$ be a continuous function. This function $f$ is said to be geometrically convex on $I$ if
\begin{equation}\label{eq2.1}
  f\bigl(x^\lambda y^{1-\lambda}\bigr)\le  f^\lambda(x) f^{1-\lambda}(y).
\end{equation}
is valid for all $x, y\in I$ and $\lambda\in[0,1]$; If the inequality~\eqref{eq2.1} is reversed, then the function $f(x)$ is said to be geometrically concave on $I$.
\end{dfn}

The notion of geometric convexity was introduced in~\cite{Jiang-Qi-conj10}. For further development on convexity of functions, please refer to~\cite{Jiang-Qi-conj11}.

\begin{lem}[{\cite[Proposition~4.3]{Jiang-Qi-conj11}}]\label{lem2.1}
Let $f:I\subset(0,\infty)\to (0,\infty)$ be a twice differentiable function. The following assertions are equivalent:
\begin{enumerate}
\item
The function $f$ is geometrically convex;
\item
The function$\frac{xf'(x)}{f(x)}$ is increasing;
\item
The function $f$ is geometrically convex if and only if
\begin{equation}\label{eq2.2}
x\bigl\{f(x)f''(x)-[f'(x)]^2\bigr\}+f(x)f'(x)\ge 0
\end{equation}
holds for all $x\in I$.
\end{enumerate}
\end{lem}

\begin{lem}[\cite{Jiang-Qi-conj12}]\label{lem2.3}
Let $f:(a,b)\subseteq(0,\infty)\to (0,\infty)$ be a geometrically concave function. Then
$$
g(x)=\int_{x}^b f(t)\td t \quad \text{and}\quad h(x)=\int_{a}^x f(t)\td t
$$
are also geometrically concave on $(a,b)$.
\end{lem}

\begin{lem}[\cite{Jiang-Qi-conj13}]\label{lem2.4}
Let $f:(a,b)\subseteq(0,\infty)\to (0,\infty)$ be monotonic and $f^{-1}$ be the inverse of $f$. Then
\begin{enumerate}
\item
if $f$ is increasing and geometrically convex \textup{(}or geometrically concave respectively\textup{)}, then $f^{-1}$ is geometrically concave \textup{(}or geometrically concave respectively\textup{)};
\item
if $f$ is decreasing and geometrically convex \textup{(}or geometrically concave respectively\textup{)}, then $f^{-1}$ is geometrically convex \textup{(}or geometrically concave respectively\textup{)}.
\end{enumerate}
\end{lem}

\section{Proof of Theorem~\ref{th1.1}}

We are now in a position to prove our main results.

\begin{proof}[Proof of Theorem~\ref{th1.1}]
Let $f(t)=(1-t^q)^{-1/p}$. Then an easy computation gives
\begin{equation}\label{eq3.1}
  \frac{tf'(t)}{f(t)}=\frac{q}{p}\cdot\frac{t^q}{1-t^q}
\end{equation}
and
\begin{equation}\label{eq3.2}
  \biggl(\frac{t^q}{1-t^q}\biggr)'=\frac{qt^{q-1}}{(1-t^q)^2}.
\end{equation}
From~\eqref{eq3.2} we see that the function $\frac{tf'(t)}{f(t)}$ is increasing. By Lemma~\ref{lem2.1}, we deduce that the function $f(t)$ is geometrically convex. By Lemma~\ref{lem2.3}, we find that the funciton $\arcsin_{p,q}(x)$ is geometrically convex.
\par
On the other hand, let $h(x)=\arcsin_{p,q}(x)$ for $x\in(0,1)$. Then
\begin{equation*}
h'(x)=(1-x^q)^{-1/p}
\end{equation*}
which is increasing. By Lemma~\ref{lem2.4}, we derive that the function $\sin_{p,q}(x)$ is geometrically concave.
\par
The rest may be proved similarly.
\end{proof}

\begin{proof}[Proof of Corollary~\ref{cor1.1}]
This follows from Theorem~\ref{th1.1} and Definition~\ref{dfn1.1}.
\end{proof}

\end{document}